\documentclass[10pt,reqno]{amsart}

\usepackage{fixltx2e}                       
\usepackage[usenames,dvipsnames]{xcolor}
\usepackage{fancyhdr}
\usepackage{amsmath,amsfonts,amsbsy,amsgen,amscd,mathrsfs,amssymb,amsthm}
\usepackage{mathtools}
\usepackage[font=small,margin=0.25in,labelfont={sc},labelsep={colon}]{caption}
\usepackage{tikz}
\usepackage{enumerate}
\usepackage{graphicx}

\usepackage{bm} 

\definecolor{dgreen}{rgb}{0.0, 0.5, 0.0}
\usepackage[colorlinks=true,linkcolor=blue,citecolor=dgreen]{hyperref}

\newtheorem{thm}{Theorem}[section]
\newtheorem{lem}[thm]{Lemma}

\newtheorem{cor}[thm]{Corollary}
\newtheorem{conj}{Conjecture}

\theoremstyle{definition}

\newtheorem{example}[thm]{Example}
\newtheorem{rem}[thm]{Remark}

\numberwithin{equation}{section} 
\numberwithin{figure}{section}
\numberwithin{table}{section}

\newcommand{\E}{\mathbf{E}}
\renewcommand{\P}{\mathbf{P}}

\begin{document}

\title[The norm of Gaussian random matrices]{On the
spectral norm of \\ Gaussian random matrices}
\author{Ramon van Handel}
\address{Sherrerd Hall Room 227, Princeton University, Princeton, NJ 
08544, USA}
\thanks{Supported in part by NSF grant
CAREER-DMS-1148711 and by the ARO through PECASE award 
W911NF-14-1-0094.}
\email{rvan@princeton.edu}

\begin{abstract} 
Let $X$ be a $d\times d$ symmetric random matrix with independent but 
non-identically distributed Gaussian entries.  It has been conjectured by 
Lata{\l}a that the spectral norm of $X$ is always of the same order as the 
largest Euclidean norm of its rows.  A positive resolution of this 
conjecture would provide a sharp understanding of the probabilistic 
mechanisms that control the spectral norm of inhomogeneous Gaussian 
random matrices.  This paper establishes the conjecture up to a 
dimensional factor of order $\sqrt{\log\log d}$.  Moreover, dimension-free 
bounds are developed that are optimal to leading order and that establish 
the conjecture in special cases. The proofs of these results shed 
significant light on the geometry of the underlying Gaussian processes. 
\end{abstract}



\dedicatory{In memory of Evarist Gin\'e}

\maketitle

\thispagestyle{empty}


\section{Introduction}
\label{sec:intro}

Let $X$ be a symmetric random matrix with independent mean zero entries.  
If the variances of the entries are all of the same order, this model is 
known as a Wigner matrix and has been widely studied in the literature 
(e.g., \cite{AGZ10}).  Due to the large amount of symmetry of such models, 
extremely precise analytic results are available on the limiting behavior 
of fine-scale spectral properties of the matrix.  Our interest, however, 
goes in an orthogonal direction. We consider the case where the variances 
of the entries are given but arbitrary: that is, we consider 
\emph{structured random matrices} where the structure is given by the 
variance pattern of the entries.  The challenge in investigating such 
matrices is to understand how the given structure of the matrix is 
reflected in its spectral properties.

In particular, we are interested in the location of the edge of the 
spectrum, that is, in the expected spectral norm $\mathbf{E}\|X\|$ of the 
matrix.  When the entries of the matrix are i.i.d., a complete 
understanding up to universal constants is provided by a remarkable result 
of Seginer \cite{Seg00} which states that the expected spectral norm of 
the matrix is of the same order as the largest Euclidean norm of its rows 
and columns.  Unfortunately, this result hinges crucially on the 
invariance of the distribution of the matrix under permutations of the 
entries, and is therefore useless in the presence of nontrivial structure.  
It is noted in \cite{Seg00} that the conclusion fails already in the 
simplest examples of structured random matrices with bounded entries.  
Surprisingly, however, no counterexamples to this statement are known for 
structured random matrices with independent \emph{Gaussian} entries.  
This observation has led to the following conjecture proposed by R.\ 
Lata{\l}a (see also \cite{Lat05,RS13}).

Throughout the remainder of this paper, $X$ will denote the $d\times d$ 
symmetric random matrix with entries $X_{ij}=b_{ij}g_{ij}$, where 
$\{g_{ij}:i\ge j\}$ are independent standard Gaussian random variables and 
$\{b_{ij}:i\ge j\}$ are given nonnegative scalars.  We write $a\lesssim b$ 
if $a\le Cb$ for a universal constant $C$, and $a\asymp b$ if $a\lesssim 
b$ and $b\lesssim a$.

\begin{conj}
\label{conj:1}
The expected spectral norm satisfies
$$
	\E\|X\| \asymp
	\E\Bigg[\max_i\sqrt{\sum_j X_{ij}^2}\Bigg].
$$
\end{conj}

The lower bound in Conjecture~\ref{conj:1} holds trivially for any 
deterministic matrix: if a matrix has a row with large Euclidean norm, 
then its spectral norm must be large.  Conjecture~\ref{conj:1} suggests 
that for Gaussian random matrices, this is the \emph{only} reason why the 
spectral norm can be large.  It is not at all clear, however, what 
mechanism might give rise to this phenomenon, particularly as the Gaussian 
nature of the entries must play a crucial role for the conjecture to hold.

Recently, Bandeira and the author \cite{BvH15} proved a sharp 
dimension-dependent upper bound on $\|X\|$ (we refer to
\cite{BvH15} for a discussion of earlier work on this topic):

\begin{thm}[\cite{BvH15}]
\label{thm:bvh}
The expected spectral norm satisfies
$$
	\E\|X\| \lesssim
	\max_i\sqrt{\sum_j b_{ij}^2} +
	\max_{ij}b_{ij}\sqrt{\log d}.
$$
\end{thm}

The combinatorial proof of this result sheds little light on the 
phenomenon described by Conjecture~\ref{conj:1}.  Nonetheless, the 
right-hand side of this expression is a natural upper bound on the 
right-hand side of Conjecture~\ref{conj:1} \cite[Remark 3.16]{BvH15}.  On 
the other hand, the terms in this bound admit another natural 
interpretation.  A simple computation shows that the first term in this 
bound is precisely $\|\E X^2\|^{1/2}$, while the second term is an upper 
bound on $\E\max_{ij}|X_{ij}|$.  This suggests the following alternative 
to Conjecture \ref{conj:1} that is also consistent with 
Theorem~\ref{thm:bvh}.

\begin{conj}
\label{conj:2}
The expected spectral norm satisfies
$$
	\E\|X\| \asymp \|\E X^2\|^{1/2} + \E\max_{ij}|X_{ij}|.
$$
\end{conj}

Once again, the lower bound in Conjecture~\ref{conj:2} holds trivially 
(cf.\ \cite[section 3.5]{BvH15}): the first term follows readily from 
Jensen's inequality, while the second term follows as the spectral norm of 
any matrix is bounded below by the magnitude of its largest entry. Thus 
the two terms in the lower bound reflect two distinct mechanisms that 
control the spectral norm of any random matrix: a random matrix has large 
spectral norm if it is large on average (as is quantified by $\|\E 
X^2\|^{1/2}$; note that the expectation here is inside the norm!), or if 
one of its entries is large (as is quantified by 
$\mathbf{E}\max_{ij}|X_{ij}|$). Conjecture~\ref{conj:2} suggests that for 
Gaussian random matrices, these are the \emph{only} reasons why the 
spectral norm can be large.

In many cases the improvement of Conjectures~\ref{conj:1} and \ref{conj:2} 
over Theorem~\ref{thm:bvh} is modest, as the latter bound is already tight 
under mild assumptions.  On the one hand, if $\max_i\sum_j b_{ij}^2 
\gtrsim \max_{ij} b_{ij}^2 \log d$, then the first term in 
Theorem~\ref{thm:bvh} dominates and therefore $\E\|X\|\asymp \|\E X^2\|^{1/2}$ as 
predicted by Conjecture~\ref{conj:2}.  On the other hand, if a polynomial 
number $\gtrsim d^c$ of entries $X_{kl}$ of the matrix have variance of 
the same order as the largest variance $b_{kl}\gtrsim \max_{ij}b_{ij}$, 
then $\E\max_{ij}|X_{ij}|\sim \max_{ij}b_{ij}\sqrt{\log d}$ and thus 
Theorem~\ref{thm:bvh} also implies Conjecture~\ref{conj:2}. These 
observations indicate that Theorem~\ref{thm:bvh} already implies 
Conjecture~\ref{conj:2} when the matrix is ``not too sparse''. 
Nonetheless, the apparent sharpness of Theorem~\ref{thm:bvh} belies a 
fundamental gap in our understanding of the probabilistic mechanisms that 
control the spectral norm of Gaussian random matrices: the phenomena 
predicted by Conjectures~\ref{conj:1} and \ref{conj:2} are inherently 
dimension-free, while the assumptions under which Theorem~\ref{thm:bvh} is 
tight exhibit nontrivial dependence on dimension. The resolution of 
Conjectures~\ref{conj:1} and \ref{conj:2} would therefore provide a 
substantially deeper insight into the structure of Gaussian random 
matrices than is obtained from Theorem~\ref{thm:bvh}.

The aim of this paper is to develop a number of new techniques and 
insights that contribute to a deeper understanding of 
Conjectures~\ref{conj:1} and \ref{conj:2}.  While our results fall short 
of resolving these conjectures, they provide strong evidence for their 
validity and shed significant light on the geometry of the problem.

We begin by observing that Conjectures~\ref{conj:1} and \ref{conj:2} are 
in fact equivalent, which is not entirely obvious at first sight.  In 
fact, our first result provides an explicit expression for the right-hand 
side in Conjectures~\ref{conj:1} and \ref{conj:2} in terms of the 
coefficients $b_{ij}$. (A much more complicated expression in terms of 
Musielak-Orlicz norms can be found in \cite{RS13}, but is too unwieldy to 
be of use in the sequel.)

\begin{thm}
\label{thm:main1}
Conjectures~\ref{conj:1} and \ref{conj:2} are equivalent:
\begin{align*}
	\E\Bigg[\max_i\sqrt{\sum_j X_{ij}^2}\Bigg]
	&\asymp
	\|\E X^2\|^{1/2} + \E\max_{ij}|X_{ij}|
	\\
	&\asymp
        \max_i\sqrt{\sum_j b_{ij}^2} +
        \max_{ij}b_{ij}^*\sqrt{\log i},
\end{align*}
where the matrix $\{b_{ij}^*\}$ is obtained by permuting the rows and 
columns of the matrix $\{b_{ij}\}$ such that $\max_jb_{1j}^*\ge
\max_j b_{2j}^*\ge\cdots\ge\max_j b_{dj}^*$.
\end{thm}

As the bound of Theorem \ref{thm:bvh} appears to be tantalizingly close to 
the expression in Theorem \ref{thm:main1}, one might hope 
that the latter could be established by a refinement of the methods that 
were developed in \cite{BvH15}.  The proof of Theorem \ref{thm:bvh} in 
\cite{BvH15} relies heavily on the moment method, which is widely used in 
the 
analysis of random matrices.  This method is based on the elementary 
observation that $\|X\|^{2p} \le \mathrm{Tr}[X^{2p}] \le d\|X\|^{2p}$ for 
any $d\times d$ symmetric matrix $X$ and $p\ge 1$, so that
$$
        \E[\|X\|^{2p}]^{1/2p} \asymp
        \E[\mathrm{Tr}[X^{2p}]]^{1/2p}\quad
        \mbox{for } p\sim\log d.
$$
The essential feature of the moment method is that the right-hand side of 
this expression is the expectation of a polynomial in the entries of the 
matrix, which admits an explicit expression that is amenable to 
combinatorial analysis.  By its very nature, any proof using the 
moment method cannot directly bound $\E\|X\|$; instead, this method bounds 
the larger quantity $\E[\|X\|^{\log d}]^{1/\log d}$, which is 
what is actually done in \cite{BvH15}. For the latter quantity, however, 
it is readily seen that the result of Theorem \ref{thm:bvh} is already 
sharp \emph{without} any additional assumptions:
$$
        \E[\|X\|^{\log d}]^{1/\log d} \asymp
	\max_i\sqrt{\sum_j b_{ij}^2} +
	\max_{ij}b_{ij}\sqrt{\log d}.
$$
The upper bound is proved in \cite{BvH15}, while the lower bound follows 
along the lines of Conjecture~\ref{conj:2}
from the estimate
$\E[\|X\|^{\log d}]^{1/\log d}\gtrsim\|\E 
X^2\|^{1/2}+\max_{ij}\|X_{ij}\|_{\log d}$.
We therefore see that the moment method is exploited optimally in 
the proof of Theorem \ref{thm:bvh}, so that the resolution of 
Conjectures~\ref{conj:1} and \ref{conj:2} \emph{cannot} be 
addressed by the same technique that gave rise to Theorem \ref{thm:bvh}.

Nonetheless, by a slicing procedure that applies Theorem~\ref{thm:bvh} 
separately at different scales, we can already establish that 
Conjectures~\ref{conj:1} and \ref{conj:2} hold up to a very mild
dimensional factor.  This is our second main result.

\begin{thm}
\label{thm:main2}
The expected spectral norm satisfies
$$
	\E\Bigg[\max_i\sqrt{\sum_j X_{ij}^2}\Bigg]
	\le
	\E\|X\| \lesssim
	\sqrt{\log\log d}~
	\E\Bigg[\max_i\sqrt{\sum_j X_{ij}^2}\Bigg].
$$
\end{thm}

While this result still exhibits an explicit dependence on dimension, the 
point of Theorem~\ref{thm:main2} is that the very mild dimensional factor 
$\sqrt{\log\log d}$ is of much smaller order than the natural scale 
$\sim\sqrt{\log d}$ that appears in the sharp dimension-dependent bound of 
Theorem~\ref{thm:bvh}; in this sense, Theorem~\ref{thm:main2} could be 
viewed as providing significant evidence for validity of 
Conjectures~\ref{conj:1} and \ref{conj:2}.

In the final part of this paper, we develop an entirely different approach 
for bounding the spectral norm of Gaussian random matrices.  Unlike the 
methods developed so far, this approach is genuinely dimension-free and 
sheds significant light on the probabilistic mechanism that lies at the 
heart of Conjectures~\ref{conj:1} and \ref{conj:2}.  The starting point 
for this approach is the elementary observation that
$$
        \E\|X\| = \E\bigg[\sup_{v\in B_2}|\langle v,Xv\rangle|\Bigg]
$$
is the expected supremum of a Gaussian process indexed by the Euclidean 
unit ball $B_2$.  It is well known that such quantities are completely 
characterized, up to universal constants, by the geometry of the 
metric space $(B_2,d)$, where 
$$
        d(v,w)^2:=\E[|\langle v,Xv\rangle-\langle w,Xw\rangle|^2]
$$
is the natural metric associated with the Gaussian process (cf.\ 
\cite{Tal14}).  Therefore, in principle, understanding the spectral norm 
of Gaussian random matrices requires ``only'' a sufficiently good 
understanding of the geometry of the metric space $(B_2,d)$.
To this end, we show that the geometry of $(B_2,d)$ can be related to
the Euclidean geometry of certain nonlinear deformations of the unit ball.
The geometric structure exhibited by this mechanism appears to almost 
resolve Conjectures~\ref{conj:1} and \ref{conj:2}, but we do not know how 
to optimally exploit this structure.  Even a crude application 
of this idea, however, suffices to prove a nontrivial dimension-free 
bound.

\begin{thm}
\label{thm:main3}
The expected spectral norm satisfies
$$
	\E\|X\|\lesssim
        \max_i\sqrt{\sum_j b_{ij}^2} +
        \max_{i}\bigg[\sum_j b_{ij}^4\bigg]^{1/4}\sqrt{\log i}.
$$
\end{thm}

It is instructive to compare this bound with the expression in
Theorem \ref{thm:main1}.  Using $2\sqrt{ab}\le a+b$, it is readily seen 
that Theorem \ref{thm:main3} implies the bound
\begin{align*}
	\E\|X\| &\lesssim
        \max_i\sqrt{\sum_j b_{ij}^2} +
        \max_{i}\bigg[\sum_j b_{ij}^2 
	\bigg]^{1/4}
	\max_{ij}\sqrt{b_{ij}^*\log i} \\ & \lesssim
        \max_i\sqrt{\sum_j b_{ij}^2} +
	\max_{ij}b_{ij}^*\log i.
\end{align*}
While this estimate falls slightly short of the conjectured optimal bound 
of Theorem~\ref{thm:main1} (due to the wrong power on the logarithm), it 
is dimension-free precisely in the expected manner.  Together with the 
natural geometric structure exhibited in the proof, this provides further 
evidence for the validity of Conjectures~\ref{conj:1} and \ref{conj:2}. 
The result of Theorem~\ref{thm:main3} is complementary to 
Theorem~\ref{thm:bvh}: while Theorem~\ref{thm:bvh} is often sharp, 
Theorem~\ref{thm:main3} can give a substantial improvement for highly 
inhomogeneous matrices.  For example, Theorem~\ref{thm:main3} readily 
implies the dimension-free bound of Lata{\l}a \cite{Lat05}, which could 
not be reproduced using Theorem~\ref{thm:bvh}.

The statement of Theorem \ref{thm:main3} was chosen for sake of 
illustration; it is in fact a direct consequence of a sharper bound that 
arises from the proof.  This sharper bound both improves somewhat on 
Theorem \ref{thm:main3} for arbitrary matrices, and is able to establish 
the validity of Conjectures~\ref{conj:1} and \ref{conj:2} in certain 
special cases.  For example, we will establish these conjectures under the 
assumption that the matrix of variances $\{b_{ij}^2\}$ is positive 
definite or has a small number of negative eigenvalues. While these 
special cases are restrictive, they emphasize that the underlying 
geometric principle is not yet exploited optimally in the proof.  The 
elimination of this inefficiency provides a promising route to the 
resolution of Conjectures \ref{conj:1} and \ref{conj:2}.

The ideas described above are developed in detail in the sequel.
Our main results, Theorems \ref{thm:main1}, \ref{thm:main2}, and 
\ref{thm:main3}, are proved 
in sections \ref{sec:gaus}, \ref{sec:slice}, and \ref{sec:geom}, 
respectively.

\section{Gaussian estimates}
\label{sec:gaus}

The aim of this section is to prove Theorem~\ref{thm:main1}.  We will, in 
fact, consider an additional quantity beside those that appear in 
Conjectures~\ref{conj:1} and \ref{conj:2}.  Let $g_1,\ldots,g_d$ be 
independent standard Gaussian variables, and consider the quantity
$$
	\E\Bigg[
	\max_i\sqrt{\sum_j b_{ij}^2g_j^2}	
	\Bigg].
$$
This quantity will appear naturally from the geometry that is to be 
developed in section~\ref{sec:geom} below.  
The maximum is taken here over random variables 
with the same distribution as in Conjecture~\ref{conj:1}
 (note that these quantities differ only in that 
$b_{ij}^2g_{j}^2$ is replaced by $X_{ij}^2=b_{ij}^2g_{ij}^2$); however, in 
the above quantity these variables are dependent, while the maximum is 
taken over independent variables in Conjecture~\ref{conj:1}.  
Nonetheless, these quantities prove to be of the same order.  The 
equivalence of the various quantities considered below indicates that the 
phenomena described by Conjectures \ref{conj:1} and \ref{conj:2} can 
appear in many different guises, providing us with substantial freedom in 
how to approach the proof of these conjectures.

\begin{thm}
\label{thm:main1ext}
The following quantities are of the same order:
\begin{align*}
	\E\Bigg[\max_i\sqrt{\sum_j X_{ij}^2}\Bigg]
	&\asymp
	\E\Bigg[\max_i\sqrt{\sum_j b_{ij}^2g_j^2}\Bigg]
	\\
	&\asymp
	\|\E X^2\|^{1/2} + \E\max_{ij}|X_{ij}|
	\\
	&\asymp
        \max_i\sqrt{\sum_j b_{ij}^2} +
        \max_{ij}b_{ij}^*\sqrt{\log i},
\end{align*}
where we recall that the matrix $\{b_{ij}^*\}$ is obtained by permuting 
the rows and columns of the matrix $\{b_{ij}\}$ such that 
$\max_jb_{1j}^*\ge\max_j b_{2j}^*\ge\cdots\ge\max_j b_{dj}^*$.
\end{thm}

\begin{rem}
\label{rem:main1sharp}
The proof of the upper bound in Theorem \ref{thm:main1ext} in fact yields
$$
	\E\Bigg[\max_i\sqrt{\sum_j b_{ij}^2g_j^2}\Bigg] \le
	\max_i\sqrt{\sum_j b_{ij}^2} + 
	C\max_{ij}b_{ij}^*\sqrt{\log(i+1)}
$$
for a universal constant $C$ (that is, the constant in front of the 
leading term is one).  This is used in section 
\ref{sec:geom} to prove Theorem \ref{thm:main3} with an optimal constant.
\end{rem}

The proof of Theorem \ref{thm:main1ext} is based on elementary estimates 
for the maxima of (sub-)Gaussian random variables with inhomogenenous 
variances.

\subsection{Gaussian maxima}

We begin by recalling a standard upper bound on the maximum of sub-Gaussian 
random variables, cf.\ \cite[Proposition 2.4.16]{Tal14}.

\begin{lem}
\label{lem:supg1}
Let $X_1,\ldots,X_n$ be not necessarily independent random variables
with
$$
        \P[X_i>x] \le C e^{-x^2/C\sigma_i^2}
        \quad\mbox{for all }x\ge 0,~ i,
$$
where $C$ is a universal constant and $\sigma_i\ge 0$ are given. Then
$$
        \E\bigg[\max_{i\le n}X_i\bigg] \lesssim
        \max_{i\le n} \sigma_i^*\sqrt{\log(i+1)},
$$
where $\sigma_1^*\ge \sigma_2^*\ge\cdots\ge\sigma_n^*$ is the decreasing
rearrangement of $\sigma_1,\ldots,\sigma_n$.
\end{lem}

The essential tool in the proof of Theorem \ref{thm:main1ext} is that the 
result of Lemma \ref{lem:supg1} can be reversed when the random variables 
are independent and Gaussian.  

\begin{lem}
\label{lem:supg2}
Let $X_1,\ldots,X_n$ be independent with
$X_i\sim N(0,\sigma_i^2)$.  Then
$$
        \E\bigg[\max_{i\le n}|X_i|\bigg] \gtrsim
        \max_{i\le n} \sigma_i^*\sqrt{\log(i+1)},
$$
where $\sigma_1^*\ge \sigma_2^*\ge\cdots\ge\sigma_n^*$ is the decreasing
rearrangement of $\sigma_1,\ldots,\sigma_n$.
\end{lem}

\begin{proof}
By permutation invariance, we can assume that 
$\sigma_i$ are nonincreasing in $i$ (so that $\sigma_i=\sigma_i^*$).
Fix $j\ge 1$ and let $g_i = X_i/\sigma_i$.  Then
$|X_i|\ge \sigma_j|g_i|$ for all $i\le j$ and $g_1,\ldots,g_j$ are 
i.i.d.\ 
standard Gaussian variables.  We therefore have
$$
	\E\bigg[\max_{i\le j}|X_i|\bigg] \ge
	\sigma_j\,\E\bigg[\max_{i\le j}|g_i|\bigg] \gtrsim
	\sigma_j\sqrt{\log(j+1)},
$$
where we used that the maximum of $j$ i.i.d.\ standard Gaussian 
variables is of order $\sqrt{\log(j+1)}$ 
\cite[Exercise 2.2.7]{Tal14}.  It remains to take the maximum over 
$j\le n$.
\end{proof}

\subsection{Proof of Theorem \ref{thm:main1ext}}

Let us begin by writing
\begin{align*}
        \E\bigg[\max_i\sqrt{\sum_j X_{ij}^2}\bigg] &\le
        \max_i\E\bigg[\sqrt{\sum_j X_{ij}^2}\bigg] \\
        &\qquad\mbox{}+ 
        \E\bigg[\max_i\bigg\{
        \sqrt{\sum_j X_{ij}^2}-\E\bigg[
        \sqrt{\sum_j X_{ij}^2}\bigg]
        \bigg\}\bigg].
\end{align*}
By Jensen's inequality, we have
$$
        \max_i\E\bigg[\sqrt{\sum_j X_{ij}^2}\bigg] \le 
	\max_i \sqrt{\sum_j b_{ij}^2}.
$$
On the other hand, by Gaussian concentration \cite[Theorem 5.8]{BLM13},
we have
$$
        \P\bigg[
        \sqrt{\sum_j X_{ij}^2}-\E\bigg[
        \sqrt{\sum_j X_{ij}^2}\bigg] > t
        \bigg]
        \le e^{-t^2/2\max_j b_{ij}^2}
$$
for every $i\le n$ and $t\ge 0$.  We therefore obtain
$$
        \E\bigg[\max_i\sqrt{\sum_j X_{ij}^2}\bigg] \le
	\max_i\sqrt{\sum_j b_{ij}^2}
        + C\max_{ij}b_{ij}^*\sqrt{\log(i+1)}
$$
by Lemma \ref{lem:supg1}, where $C$ is a universal constant.
As
$$
	\max_{ij}b_{ij}^*\sqrt{\log(i+1)} \lesssim
	\max_{j}b_{1j}^* +
	\max_{ij}b_{ij}^*\sqrt{\log i} \lesssim
	\max_i\sqrt{\sum_j b_{ij}^2} + \max_{ij}b_{ij}^*\sqrt{\log i},
$$
we have shown
$$
        \E\bigg[\max_i\sqrt{\sum_j X_{ij}^2}\bigg] \lesssim
	\max_i\sqrt{\sum_j b_{ij}^2}
        + \max_{ij}b_{ij}^*\sqrt{\log i}
$$
(this last step is irrelevant to our results and is included for 
cosmetic reasons only).

Next, we note that
\begin{align*}
	\sum_j b_{ij}^2 &=
	\E\bigg[\sqrt{\sum_j X_{ij}^2}\bigg]^2 +
	\mathrm{Var}\bigg[\sqrt{\sum_j X_{ij}^2}\bigg] \\
	&\le \E\bigg[\sqrt{\sum_j X_{ij}^2}\bigg]^2 +
	\max_j b_{ij}^2
	\lesssim \E\bigg[\sqrt{\sum_j X_{ij}^2}\bigg]^2,
\end{align*}
where we have used the Gaussian Poincar\'e inequality \cite[Theorem 
3.20]{BLM13}.  Therefore,
$$
	\E\bigg[\max_i\sqrt{\sum_j X_{ij}^2}\bigg] \gtrsim
	\max_i\sqrt{\sum_j b_{ij}^2} = \|\E X^2\|^{1/2}.
$$
On the other hand, we trivially have
$$
	\E\bigg[\max_i\sqrt{\sum_j X_{ij}^2}\bigg] \ge
	\E\max_{ij}|X_{ij}|.
$$
Averaging these bounds gives
$$
	\|\E X^2\|^{1/2} + \E\max_{ij}|X_{ij}| 
	\lesssim
	\E\bigg[\max_i\sqrt{\sum_j X_{ij}^2}\bigg].
$$
In the opposite direction, for every $i$, choose $j(i)$ such
that $b_{ij(i)} = \max_j b_{ij}$.  Then
$$
	\E\max_{ij}|X_{ij}| \ge
	\E\max_i|X_{ij(i)}| \gtrsim
	\max_{ij}b_{ij}^*\sqrt{\log i}
$$
by Lemma \ref{lem:supg2}.
Putting together the above bounds, we have shown that
\begin{align*}
	\max_i\sqrt{\sum_j b_{ij}^2} +
	\max_{ij}b_{ij}^*\sqrt{\log i} &\lesssim
	\|\E X^2\|^{1/2} + \E\max_{ij}|X_{ij}|
	\\
	&\lesssim
        \E\bigg[\max_i\sqrt{\sum_j X_{ij}^2}\bigg] \\
	&\lesssim
	        \max_i\sqrt{\sum_j b_{ij}^2} +
        \max_{ij}b_{ij}^*\sqrt{\log i}.
\end{align*}
This establishes the equivalence between Conjectures~\ref{conj:1} and
\ref{conj:2}.

It remains to consider the second quantity in Theorem \ref{thm:main1ext}.
The upper bound
\begin{align*}
        \E\bigg[\max_i\sqrt{\sum_j b_{ij}^2g_j^2}\bigg] &\le
        \max_i\sqrt{\sum_j b_{ij}^2}
        + C\max_{ij}b_{ij}^*\sqrt{\log(i+1)} \\
	&\lesssim
	\max_i\sqrt{\sum_j b_{ij}^2} + \max_{ij}b_{ij}^*\sqrt{\log i}
\end{align*}
and the lower bound
$$
	\E\bigg[\max_i\sqrt{\sum_j b_{ij}^2g_j^2}\bigg]
	\gtrsim \max_i\sqrt{\sum_j b_{ij}^2}
$$
are obtained by repeating \emph{verbatim} the corresponding arguments for 
the first quantity in Theorem \ref{thm:main1ext}.  On the other hand, we 
can now estimate
$$
	\E\bigg[\max_i\sqrt{\sum_j b_{ij}^2g_j^2}\bigg] \ge
	\E\max_{ij}b_{ij}|g_j| \gtrsim
	\max_{ij}b_{ij}^*\sqrt{\log i}
$$
by Lemma \ref{lem:supg2}.  Averaging these bounds
completes the proof. \qed

\section{Slicing}
\label{sec:slice}

The aim of this short section is to prove Theorem \ref{thm:main2}.  The 
lower bound is trivial, and therefore by Theorem \ref{thm:main1} it 
remains to prove the following.

\begin{thm}
\label{thm:main2ext}
The expected spectral norm satisfies
$$
	\E\|X\| \lesssim
	\sqrt{\log\log d}\,
	\Bigg(
        \max_i\sqrt{\sum_j b_{ij}^2} +
        \max_{ij}b_{ij}^*\sqrt{\log i}
	\Bigg).
$$
\end{thm}

This result will be established by slicing the matrix into $\sim\log\log d$
pieces at different scales, each of which is bounded separately using
Theorem \ref{thm:bvh}.

It proves to be convenient for the present purposes to work with matrices 
with independent entries that are not symmetric (as opposed to symmetric 
matrices, for which $X_{ij}=X_{ji}$ are not independent).  To this end, 
let us cite the following non-symmetric variant of Theorem \ref{thm:bvh}, 
see \cite[Theorem 3.1]{BvH15} and its proof.

\begin{thm}[\cite{BvH15}]
\label{thm:bvhrect}
Let $Z$ be the $d_1\times d_2$ matrix whose entries $Z_{ij}\sim 
N(0,c_{ij}^2)$ are independent Gaussian variables.  Then the expected
spectral norm satisfies
$$
	\E[\|Z\|^2]^{1/2} \lesssim	
	\max_i\sqrt{\sum_j c_{ij}^2} +
	\max_j\sqrt{\sum_i c_{ij}^2} +
	\max_{ij}c_{ij}\sqrt{\log(d_1\wedge d_2)}.
$$
\end{thm}

We can now proceed to the proof of Theorem \ref{thm:main2ext}.

\begin{proof}[Proof of Theorem \ref{thm:main2ext}]
By permuting the rows and columns of $X$ if necessary, we can assume 
without loss of generality in the sequel that $b_{ij}=b_{ij}^*$.

We begin by decomposing the matrix $X=X^\uparrow + X^\downarrow$ into its 
parts above and below the diagonal: that is, $X^\uparrow_{ij} :=
X_{ij}\mathbf{1}_{i<j}$ and $X^\downarrow := X_{ij}\mathbf{1}_{i\ge j}$.
As
$$
	\E\|X\| \le \E\|X^\uparrow\|+\E\|X^\downarrow\| \le
	2\,\E\|X^\downarrow\|
$$
(the second bound follows by Jensen's inequality), it suffices
to bound $\E\|X^\downarrow\|$.

We now decompose $X^\downarrow$ into $N:=\lceil\log_2\log_2 d\rceil$
horizontal slices as follows:
$$
	X^\downarrow = \sum_{n=1}^N X^{(n)}
$$
with
$$
	X^{(1)}_{ij} := X^\downarrow_{ij}\mathbf{1}_{i\le 4},\qquad
	X^{(n)}_{ij} := X^\downarrow_{ij}\mathbf{1}_{2^{2^{n-1}}<i\le
	2^{2^n}}\mbox{ for }2\le n\le N.
$$
Each matrix $X^{(n)}$ has independent entries, and the only nonzero 
entries of this matrix are contained in its upper $2^{2^n}\times 2^{2^n}$ 
block.  Moreover,
$$
	\|X^\downarrow\|^2 =
	\|X^{\downarrow *}X^\downarrow\|=
	\Bigg\|\sum_{n=1}^NX^{(n)*}X^{(n)}\Bigg\| \le
	\sum_{n=1}^N \|X^{(n)}\|^2.
$$
We therefore have
$$
	\E\|X\| \le
	2\sqrt{N}
	\max_{n\le N}\E[\|X^{(n)}\|^2]^{1/2}.
$$
We now apply Theorem \ref{thm:bvhrect} to estimate each term
$\E[\|X^{(n)}\|^2]$.  Define the quantities
$$
	\sigma := \max_i\sqrt{\sum_j b_{ij}^2},\qquad\quad
	\Gamma := \max_{ij}b_{ij}^*\sqrt{\log i}.
$$
As we assumed that $b_{ij}=b_{ij}^*$, it follows immediately that
$$
	b_{ij} \le \frac{\Gamma}{\sqrt{\log i}}
	\quad\mbox{for all }i,j.
$$
In particular, this implies that for $n\ge 2$
$$
	\mathrm{Var}(X^{(n)}_{ij}) 
	\le \frac{\Gamma^2}{\log 2^{2^{n-1}}}
	\lesssim
	2^{-n}\Gamma^2\quad\mbox{for all }i,j.
$$
On the other hand, the sum of the variances of the entries in any row or 
column of $X^{(n)}$ is clearly still bounded by $\sigma^2$.  Finally, as 
noted above, $X^{(n)}$ is a ${2^{2^n}\times 2^{2^n}}$-dimensional matrix
(we can remove all vanishing rows and columns without decreasing the 
norm).  Applying Theorem \ref{thm:bvhrect} yields for every $2\le n\le N$
$$
	\E[\|X^{(n)}\|^2]^{1/2} \lesssim
	\sigma + 2^{-n/2}\Gamma\sqrt{\log 2^{2^n}} \lesssim
	\sigma + \Gamma.
$$
On the other hand, applying Theorem \ref{thm:bvhrect} with
$d_1=d_2=4$ immediately yields the analogous bound for
$X^{(1)}$.  We therefore finally obtain
$$
        \E\|X\| \lesssim 
        \sqrt{N}(\sigma+\Gamma),
$$
which completes the proof.
\end{proof}

The proof of Theorem \ref{thm:main2ext} does not really contain a new 
idea: it follows directly from the dimension-dependent bound of Theorem 
\ref{thm:bvhrect} by applying it in a multiscale fashion.  The problem 
with this approach is that while we engineered the slices so that 
Theorem \ref{thm:bvhrect} is sharp on each slice, substantial loss 
is incurred in the estimate
$$
	\|X^\downarrow\|^2 \le
	\sum_{n=1}^N \|X^{(n)}\|^2,
$$
that is, when we assemble the slices to obtain the final bound. To 
illustrate this loss, consider the case where $X$ is a diagonal matrix 
with $b_{ii}=(\log(i+1))^{-1/2}$. Then it is easily seen that in fact 
$\|X^\downarrow\|^2= \max_{n}\|X^{(n)}\|^2$, while every term 
$\|X^{(n)}\|^2$ is of comparable magnitude. We therefore see in this 
example that the residual dimension-dependence in Theorem 
\ref{thm:main2} is incurred entirely in the above estimate.

Notice that in contrast to the above estimate, we have the exact identity
$$
	\|X^\downarrow\|^2 =
	\sup_{v\in B_2}
	\sum_{n=1}^N \|X^{(n)}v\|^2.
$$
The previous estimate is sharp when each term in the sum is 
simultaneously maximized by the same vector $v$. As the 
matrices $X^{(n)}$ are independent and have vastly different dimensions 
and scales, it seems particularly unlikely that this will be the case. 
If it were possible to show that in fact 
$\|X^\downarrow\|^2\approx\max_{n}\|X^{(n)}\|^2$ holds in the general setting, 
then the slicing method could be adapted to prove Conjectures \ref{conj:1} 
and \ref{conj:2}. However, it is far from clear how this idea could be 
made precise, and it appears that the residual dimension-dependence in 
Theorem \ref{thm:main2} cannot be further reduced without the 
introduction of a genuinely new idea.

\section{Geometry}
\label{sec:geom}

The aim of this section is to exhibit a very useful mechanism to control 
the geometric structure of the Gaussian processes associated to Gaussian 
random matrices.  A direct application of this mechanism gives rise to 
dimension-free bounds on the spectral norm of Gaussian random matrices 
that can improve significantly on Theorem \ref{thm:bvh} for highly 
inhomogeneous matrices. Let us begin by formulating a general result that 
can be obtained by this method, from which Theorem \ref{thm:main3} and a 
number of other interesting consequences will follow as corollaries.

\subsection{A general result}

In the sequel, we will denote by $B$ the $d\times d$ symmetric matrix 
of variances of the entries of $X$, that is, $B_{ij} := b_{ij}^2$.  We 
denote by $B^+$ and $B^-$ its positive and negative parts, respectively; 
that is, if $B=\sum_i \lambda_i u_iu_i^*$ is the spectral decomposition of 
$B$, then $B^+:=\sum_i (\lambda_i\vee 0)u_iu_i^*$ and 
$B^-:= -\sum_i (\lambda_i\wedge 0)u_iu_i^*$.

\begin{thm}
\label{thm:main3ext}
Let $Y\sim N(0,B^-)$ be Gaussian with covariance matrix
$B^-$, and let $g_1,\ldots,g_d\sim N(0,1)$ be i.i.d.\ standard 
Gaussian variables.  Then for any $\gamma>0$
$$
	\E\|X\| \le
	\sqrt{2+\gamma+\gamma^{-1}}~
	\E\Bigg[\max_i\sqrt{\sum_j b_{ij}^2g_j^2}\Bigg] +
	\sqrt{\gamma}~\E\Big[\max_{i}Y_i\Big]+2\max_{ij}b_{ij}.
$$
\end{thm}

As a first consequence, we deduce a sharp form of Theorem \ref{thm:main3}.

\begin{cor}
\label{cor:main3sharp}
There is a universal constant $C$ such that
$$
	\E\|X\| \le 2\max_i\sqrt{\sum_j b_{ij}^2} +
	C\max_{i}\bigg[\sum_j b_{ij}^4\bigg]^{1/4}\sqrt{\log(i+1)}.
$$
\end{cor}

\begin{proof}
As $B^2 = (B^+)^2+(B^-)^2$, we have
$$
	\sum_j (B^+_{ij})^2 +
	\sum_j (B^-_{ij})^2 =
	\sum_j B_{ij}^2 = 
	\sum_j b_{ij}^4.
$$
Therefore,
$$
	\mathrm{Var}(Y_i) =
	B^-_{ii} \le
	\sqrt{\sum_j b_{ij}^4}
$$
for every $i$, and Lemma \ref{lem:supg1} gives
$$
	\E\Big[\max_i Y_i\Big] \lesssim
	\max_{i}\bigg[\sum_j b_{ij}^4\bigg]^{1/4}\sqrt{\log(i+1)}.
$$
On the other hand, by Remark \ref{rem:main1sharp}, we have
\begin{align*}
        \E\Bigg[\max_i\sqrt{\sum_j b_{ij}^2g_j^2}\Bigg] &\le
        \max_i\sqrt{\sum_j b_{ij}^2} +
        C'\max_{ij}b_{ij}^*\sqrt{\log(i+1)} \\
	&\le
        \max_i\sqrt{\sum_j b_{ij}^2} +
	C'\max_{i}\bigg[\sum_j b_{ij}^4\bigg]^{1/4}\sqrt{\log(i+1)}
\end{align*}
for a universal constant $C'$.  Now apply Theorem 
\ref{thm:main3ext} with $\gamma=1$.
\end{proof}

Let us note that the leading term in the first inequality of 
Corollary~\ref{cor:main3sharp} is sharp.  To see this, consider the 
example of a Wigner matrix where $b_{ij}=1$ for all $i,j$.  Then the first 
inequality yields $\E\|X\| \le 2\sqrt{d} + o(d)$, which precisely matches 
the exact asymptotic $\|X\| \sim 2\sqrt{d}$ as $d\to\infty$ \cite[Theorem 
2.1.22]{AGZ10}.  On the other hand, the second term in this inequality is 
suboptimal, as can be seen by considering the example where $B$ is a band 
matrix with $b_{ij}=1$ inside a diagonal band of width $\sim\sqrt{\log d}$ 
and $b_{ij}=0$ outside the band (compare the conclusion of 
Corollary~\ref{cor:main3sharp} with that of Theorem~\ref{thm:bvh}).  In 
cases such as the latter example where the second term dominates, 
Corollary~\ref{cor:main3sharp} can be improved slightly by optimizing over
$\gamma$.

\begin{cor}
\label{cor:main3opt}
The expected spectral norm satisfies
\begin{multline*}
	\E\|X\| \lesssim
	\max_i\sqrt{\sum_j b_{ij}^2} + 
	\max_{ij}b_{ij}^*\sqrt{\log i} + \mbox{} \\
	\Bigg[
	\max_i\sqrt{\sum_j b_{ij}^2}
	+\max_{ij}b_{ij}^*\sqrt{\log i}
	\Bigg]^{1/2}\Bigg[
	\max_{i}\bigg[\sum_j b_{ij}^4\bigg]^{1/4}\sqrt{\log i}
	\Bigg]^{1/2}.
\end{multline*}
\end{cor}

\begin{proof}
Apply Theorems \ref{thm:main3ext} and \ref{thm:main1} and optimize 
over $\gamma>0$.   
\end{proof}

Despite the suboptimal nature of the second term in 
Corollaries~\ref{cor:main3sharp} and \ref{cor:main3opt}, these results can 
improve significantly on Theorem \ref{thm:bvh} for highly inhomogeneous 
matrices.  To illustrate this, let us use Corollary \ref{cor:main3sharp} 
to derive a delicate (but much less sharp) result of Lata{\l}a 
\cite{Lat05} that could not be recovered from Theorem \ref{thm:bvh}.

\begin{cor}[\cite{Lat05}]
The expected spectral norm satisfies
$$
	\E\|X\| \lesssim
	\max_i\sqrt{\sum_j b_{ij}^2} + 
	\sqrt[4]{\sum_{ij} b_{ij}^4}.
$$
\end{cor}

\begin{proof}
We may assume without loss of generality that the rows and columns of $X$ 
have been ordered such that $\sum_j b_{ij}^4$ is nonincreasing in $i$.  
Then we must have
$$
	\sum_j b_{ij}^4 \le
	\frac{1}{i}\sum_{ij} b_{ij}^4
$$
for all $i$, and the conclusion follows readily from Corollary 
\ref{cor:main3sharp}.
\end{proof}

The above corollaries are based on a rather crude estimate on the variance 
of the random variables $Y_i$ that appear in Theorem~\ref{thm:main3ext} 
(see the proof of Corollary~\ref{cor:main3sharp}).  Unfortunately, it 
seems that this estimate cannot be significantly improved in general, 
which indicates that there is genuine inefficiency in the proof of 
Theorem~\ref{thm:main3ext}.  The apparent origin of this inefficiency will 
be discussed in some detail in the sequel.  It is interesting to note, 
however, that there are certain special cases where 
Theorem~\ref{thm:main3ext} already provides substantially better results 
than is suggested by Corollary~\ref{cor:main3sharp}.  For example, 
Theorem~\ref{thm:main3ext} immediately resolves 
Conjecture~\ref{conj:1} (with optimal constant!) under the strong 
assumption that the matrix of variances $B$ is positive semidefinite.

\begin{cor}
\label{cor:posdef}
If $B$ is positive semidefinite, then
$$
	\E\|X\| \le 2\,\E\Bigg[\max_i\sqrt{\sum_j b_{ij}^2g_j^2}\Bigg]
	+ 2\max_{ij}b_{ij}.
$$
\end{cor}

\begin{proof}
In this case $B^-=0$, and the result follows from
Theorem \ref{thm:main3ext} with $\gamma=1$.
\end{proof}

Along similar lines, it is not difficult to see that if $B$ has at most 
$k$ negative eigenvalues, than the conclusion of Conjecture~\ref{conj:1} 
holds with a constant that depends on $k$ only (so that the conjecture is 
established if $B$ has $O(1)$ negative eigenvalues).  On the other hand, 
there are other cases where the special structure of $B$ makes it possible 
to deduce Conjecture~\ref{conj:1} from Theorem \ref{thm:main3ext}.  For 
example, if 
$$
	B=B'\otimes\begin{pmatrix} 0 & 1 \\ 1 & 0\end{pmatrix}
$$
where $B'$ is a positive semidefinite matrix, then
$$
	B^- = B'\otimes\frac{1}{2}\begin{pmatrix} 1 & -1 \\ -1 & 
		1\end{pmatrix},
$$
so that $\mathrm{Var}(Y_i) \le \frac{1}{2}\max_j b_{ij}^2$ for all $i$; 
then arguing as in the proof of Corollary~\ref{cor:main3sharp} and 
applying Theorem~\ref{thm:main1} immediately yields 
Conjecture~\ref{conj:1}.  All of these special cases are restrictive; 
however, they emphasize that the approach developed in this 
section can already extend significantly beyond the result of Theorem 
\ref{thm:main3}.

\subsection{Proof of Theorem \ref{thm:main3ext}}

We begin by recalling that
$$
        \E\|X\| = \E\bigg[\sup_{v\in B_2}|\langle v,Xv\rangle|\Bigg]
$$
is the expected supremum of a Gaussian process indexed by the Euclidean 
unit ball $B_2$.  It is well known that the supremum of a Gaussian 
process is intimately connected with the geometry defined by the 
associated (semi)metric
$$
        d(v,w)^2:=\E[|\langle v,Xv\rangle-\langle w,Xw\rangle|^2].
$$
The difficulty we face is to understand how to control this rather 
strange geometry.

To motivate the device that we will use for this purpose, let us disregard 
for the moment the natural metric $d$ and consider instead a simpler 
quantity, the variance of the Gaussian process.  We can easily compute
$$
	\E[\langle v,Xv\rangle^2] =
	2\sum_{i\ne j} v_i^2b_{ij}^2v_j^2 +
	\sum_i b_{ii}^2 v_i^4
	\le
	2\sum_{ij} v_i^2b_{ij}^2v_j^2.
$$
We now observe that this expression can be reorganized in a suggestive
manner.  Define the norm $\|\cdot\|_i$ on $\mathbb{R}^d$ and the nonlinear 
map $x:\mathbb{R}^d\to\mathbb{R}^d$ as
$$
	\|v\|_i^2 := \sum_j b_{ij}^2v_j^2,\qquad\quad
	x_i(v) := v_i\|v\|_i,
$$
and consider a second Gaussian process
$$
	\langle x(v),g\rangle :=
	\sum_i x_i(v)g_i
$$
where $g_1,\ldots,g_d$ are i.i.d.\ standard Gaussian variables.  Then
$$
	\E[\langle x(v),g\rangle^2] = 
	\|x(v)\|^2 =
	\sum_{ij} v_i^2b_{ij}^2v_j^2.
$$
In particular, we see that the variance of the Gaussian process 
$\{\langle v,Xv\rangle\}_{v\in B_2}$ associated with our random matrix
is dominated up to a constant by the variance of the Gaussian process
$\{\langle x(v),g\rangle\}_{v\in B_2}$.  The latter process is precisely 
what we would like to obtain in our upper bound, as we immediately 
compute
$$
	\sup_{v\in B_2}\langle x(v),g\rangle \le 
	\sup_{v\in B_2}\sqrt{\sum_{ij}v_i^2b_{ij}^2g_j^2} =
	\max_i\sqrt{\sum_j b_{ij}^2g_j^2}
$$
using the Cauchy-Schwarz inequality and the fact that the map $v\mapsto
(v_i^2)$ maps the Euclidean unit ball in $\mathbb{R}^d$ onto the 
$d$-dimensional simplex.

Unfortunately, an inequality between the variances of Gaussian 
processes does not suffice to control the suprema of these processes.  
What is sufficient, however, is to establish such an inequality 
between the natural distances of these Gaussian processes: if we could 
show that the natural distance of the Gaussian process 
$\{\langle v,Xv\rangle\}_{v\in B_2}$ is dominated by the natural distance 
of $\{\langle x(v),g\rangle\}_{v\in B_2}$, that is,
$$
	d(v,w) \,\stackrel{?}{\lesssim}\,
	\|x(v)-x(w)\|,
$$
then the conclusion of Conjecture \ref{conj:1} would follow immediately 
from the Slepian-Fernique lemma \cite[Theorem 13.3]{BLM13}.
Unfortunately, this inequality does not always hold, see section 
\ref{sec:disc} below.  However, it turns out that such an inequality 
\emph{nearly} holds, and this is the key device that will be exploited in 
this section.

\begin{lem}[The basic principle]
\label{lem:basic}
For every $v,w\in\mathbb{R}^d$ and $\gamma>0$
$$
	d(v,w)^2 \le
	(2+\gamma+\gamma^{-1})\,\|x(v)-x(w)\|^2 -
	\gamma
	\sum_{ij}(v_i^2-w_i^2)b_{ij}^2(v_j^2-w_j^2).
$$
\end{lem}

\begin{proof}
We first compute $d(v,w)$:
\begin{align*}
        d(v,w)^2 &=
        \E[\langle v+w,X(v-w)\rangle^2] 
	\phantom{\sum_i} \\
        &=
        \sum_i (v_i^2-w_i^2)^2b_{ii}^2 +
        \sum_{i>j} \{(v_i+w_i)(v_j-w_j)+(v_i-w_i)(v_j+w_j)\}^2b_{ij}^2
        \\
        &=
        \sum_{ij} (v_i+w_i)^2b_{ij}^2(v_j-w_j)^2 + 
        \sum_{i\ne j} (v_i^2-w_i^2)b_{ij}^2(v_j^2-w_j^2).
\end{align*}
We can now estimate using the triangle inequality $\|v+w\|_i\le 
\|v\|_i+\|w\|_i$
\begin{align*}
        d(v,w)^2 &=
        \sum_{i} (v_i-w_i)^2\|v+w\|_i^2 +
        \sum_{i\ne j} (v_i^2-w_i^2)b_{ij}^2(v_j^2-w_j^2) \\
        &\le
        \sum_{i} (v_i-w_i)^2(\|v\|_i+\|w\|_i)^2 + 
        \sum_{ij} (v_i^2-w_i^2)b_{ij}^2(v_j^2-w_j^2) \\
        &=
        \sum_{i} (v_i-w_i)^2(\|v\|_i+\|w\|_i)^2 \\
        &\qquad \mbox{}+
        \sum_{i} (v_i-w_i)(\|v\|_i+\|w\|_i)
                (v_i+w_i)(\|v\|_i-\|w\|_i) \\
        &=
        2\sum_{i} (v_i-w_i)(\|v\|_i+\|w\|_i)(v_i\|v\|_i-w_i\|w\|_i) \\
	&=
	2\|x(v)-x(w)\|^2 +
	2\sum_i (v_i\|w\|_i-w_i\|v\|_i)(v_i\|v\|_i-w_i\|w\|_i).
\end{align*}
The elementary inequality $2ab\le \gamma a^2+\gamma^{-1}b^2$ gives
$$
	d(v,w)^2 \le
	(2+\gamma^{-1})\|x(v)-x(w)\|^2 +
	\gamma\sum_i (v_i\|w\|_i-w_i\|v\|_i)^2
$$
for any $\gamma>0$.  We now compute
\begin{align*}
	&\sum_i (v_i\|w\|_i-w_i\|v\|_i)^2 \\
	&\quad=
	2\sum_{ij} v_i^2b_{ij}^2w_j^2 -
	2\langle x(v),x(w)\rangle \\
	&\quad=
	\|x(v)-x(w)\|^2 +
	2\sum_{ij} v_i^2b_{ij}^2w_j^2 -
	\sum_{ij} v_i^2b_{ij}^2v_j^2 -
	\sum_{ij} w_i^2b_{ij}^2w_j^2 \\
	&\quad=
	\|x(v)-x(w)\|^2 -
	\sum_{ij} (v_i^2-w_i^2)b_{ij}^2(v_j^2-w_j^2).
\end{align*}
Combining these bounds completes the proof.
\end{proof}

With Lemma \ref{lem:basic} in hand, we can now easily complete the
proof of Theorem \ref{thm:main3ext}.

\begin{proof}[Proof of Theorem \ref{thm:main3ext}]
We begin by noting that the spectral norm of a symmetric matrix
is the largest magnitude of its maximal and minimal eigenvalues, that is,
$$
	\|X\| = \sup_{v\in B_2}|\langle v,Xv\rangle|
	=\sup_{v\in B_2}\langle v,Xv\rangle \vee
	\sup_{v\in B_2}\langle v,(-X)v\rangle.
$$
As $X$ and $-X$ have the same distribution, we can estimate
\begin{align*}
	\E\|X\| &\le
	\E\bigg[\sup_{v\in B_2}\langle v,Xv\rangle\bigg] +
	\sqrt{2}\,\mathrm{Var}\bigg[
	\sup_{v\in B_2}\langle v,Xv\rangle
	\bigg]^{1/2}
	\\
	&\le \E\bigg[\sup_{v\in B_2}\langle v,Xv\rangle\bigg] +
	2\max_{ij}b_{ij},
\end{align*}
where we used the Gaussian Poincar\'e inequality 
\cite[Theorem 3.20]{BLM13} in the second inequality.
To proceed, assume without loss of generality that $Y\sim N(0,B^-)$ is
independent of $g_1,\ldots,g_d$, and define the Gaussian process
$\{Z_v\}_{v\in B_2}$ as follows:
$$
	Z_v := \sqrt{2+\gamma+\gamma^{-1}}\,\langle x(v),g\rangle + 
	\sqrt{\gamma}\sum_i v_i^2 Y_i.
$$
The natural distance of this Gaussian process satisfies
$$
	\E[|Z_v-Z_w|^2] =
	(2+\gamma+\gamma^{-1})\,\|x(v)-x(w)\|^2 +
	\gamma\sum_{ij} (v_i^2-w_i^2)B^-_{ij}(v_j^2-w_j^2)
	\ge d(v,w)^2
$$
by Lemma \ref{lem:basic}.  We therefore obtain
$$
	\E\bigg[\sup_{v\in B_2}\langle v,Xv\rangle\bigg]
	\le
	\E\bigg[
	\sup_{v\in B_2}Z_v
	\bigg]
$$
by the Slepian-Fernique inequality \cite[Theorem 13.3]{BLM13}.
A simple application of the Cauchy-Schwarz inequality as discussed
before Lemma \ref{lem:basic} completes the proof.
\end{proof}

\subsection{Discussion}
\label{sec:disc}

It is instructive to discuss the geometric significance of the basic 
principle described by Lemma~\ref{lem:basic}.  The clearest illustration 
of this device appears in the setting of Corollary \ref{cor:posdef} 
where the matrix of variances $B$ is positive semidefinite.  In this case, 
the second term in Lemma~\ref{lem:basic} is nonpositive, and we obtain
$$
	d(v,w) \le 2\|x(v)-x(w)\|.
$$
This inequality maps the geometry of the metric space 
$(B_2,d)$ onto the Euclidean geometry of the nonlinear 
deformation of the unit ball
$$
        B_* := \{x(v):v\in B_2\},
$$
which is much easier to understand.

\begin{example}
The trivial case of this construction appears in the example of a Wigner 
matrix where $b_{ij}=1$ for all $i,j$.  In this special case, the 
nonlinear deformation has no effect and $B_*=B_2$ is simply the Euclidean 
unit ball.  Applying the Slepian-Fernique inequality in this setting shows that
$$
	\E\|X\|\lesssim \E\bigg[\sup_{t\in B_*}\langle t,g\rangle\bigg] =
	\E\|g\|\le \sqrt{d}.
$$
This idea is not new: our approach reduces in this trivial 
setting to the well-known method of Gordon for estimating the norm of 
Wigner matrices \cite[section 5.3.1]{Ver12}.  However, the 
crucial insight developed here is that the geometry of $B_*$ changes 
drastically when we depart from the simple setting of Wigner matrices, 
which is not captured by Gordon's method.  This is illustrated in the 
following example.
\end{example}

\begin{example}
Consider the example of a diagonal random matrix where 
$b_{ij}=\mathbf{1}_{i=j}$.  In this case, the nonlinear deformation 
transforms the Euclidean unit ball into the $\ell_1$-ball $B_*=B_1$,
whose geometry is entirely different than in the previous example.
Applying the Slepian-Fernique inequality in this setting shows that
$$
	\E\|X\|\lesssim \E\bigg[\sup_{t\in B_*}\langle t,g\rangle\bigg] =
	\E\|g\|_\infty \lesssim \sqrt{\log d},
$$
which captures precisely the correct behavior in this setting.
\end{example}

\begin{figure}
\begin{tikzpicture}[scale=.9]
\tikzset{samples=50}

\draw (1.5,2) node {$\scriptstyle B =
        \left(\begin{smallmatrix} 1\phantom{/} & 0\\ 0\phantom{/} & 
1\end{smallmatrix}\right)$};
\draw[thick,fill=blue!25, domain=0:2*pi] plot
      ({1.5+cos(\x r)*sqrt(cos(\x r)*cos(\x r)+0*sin(\x r)*sin(\x r))},
       {sin(\x r)*sqrt(sin(\x r)*sin(\x r)+0*cos(\x r)*cos(\x r))});
\draw[thick] (0,0) to (3,0);
\draw[thick] (1.5,-1.5) to (1.5,1.5);

\draw (5,2) node {$\scriptstyle B=
        \left(\begin{smallmatrix} 1 & 1/2\\ 1/2 & 
1\end{smallmatrix}\right)$};
\draw[thick,fill=blue!25, domain=0:2*pi] plot
      ({5+cos(\x r)*sqrt(cos(\x r)*cos(\x r)+sin(\x r)*sin(\x r)/2)},
       {sin(\x r)*sqrt(sin(\x r)*sin(\x r)+cos(\x r)*cos(\x r)/2)});
\draw[thick] (3.5,0) to (6.5,0);
\draw[thick] (5,-1.5) to (5,1.5);

\draw (8.5,2) node {$\scriptstyle B=
        \left(\begin{smallmatrix} 1\phantom{/} & 1\\ 1\phantom{/} & 
1\end{smallmatrix}\right)$};
\draw[thick,fill=blue!25, domain=0:2*pi] plot
      ({8.5+cos(\x r)*sqrt(cos(\x r)*cos(\x r)+sin(\x r)*sin(\x r))},
       {sin(\x r)*sqrt(sin(\x r)*sin(\x r)+cos(\x r)*cos(\x r))});
\draw[thick] (7,0) to (10,0);
\draw[thick] (8.5,-1.5) to (8.5,1.5);

\draw (12,2) node {$\scriptstyle B=
        \left(\begin{smallmatrix} 1/8 & 1\\ 1 & 
1/8\end{smallmatrix}\right)$};
\draw[thick,fill=blue!25, domain=0:2*pi] plot
      ({12+cos(\x r)*sqrt(cos(\x r)*cos(\x r)/8+sin(\x r)*sin(\x r))},
       {sin(\x r)*sqrt(sin(\x r)*sin(\x r)/8+cos(\x r)*cos(\x r))});
\draw[thick] (10.5,0) to (13.5,0);
\draw[thick] (12,-1.5) to (12,1.5);

\end{tikzpicture}
\caption{Various possible shapes of the deformed ball $B_*=\{x(v):v\in 
B_2\}$ are illustrated in the two-dimensional case $d=2$.
Note that the matrix $B$ is positive semidefinite in the first three 
examples but not in the fourth example.
\label{fig:interp}}
\end{figure}
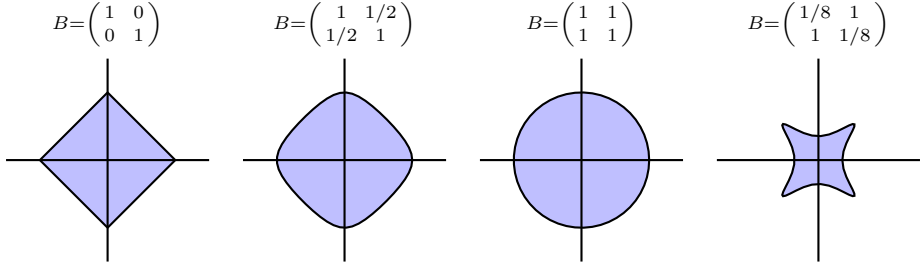
In general, the deformed ball $B_*$ can take very different shapes, as 
is illustrated in Figure \ref{fig:interp}.  The beauty of this 
construction is that the manner in which the geometry of the 
space $(B_2,d)$ is captured by the geometry of $(B_*,\|\cdot\|)$
provides a clear mechanism that gives rise to the phenomenon predicted by 
Conjecture~\ref{conj:1}.

Unfortunately, the simple geometry exhibited above is much 
less clear when the matrix $B$ is not positive semidefinite.
One might hope that the inequality 
$$
	d(v,w) \,\stackrel{?}{\lesssim}\,
	\|x(v)-x(w)\|
$$
remains valid in the general setting, but this is not always true.
The following illuminating example was suggested by Afonso Bandeira.

\begin{example}
\label{ex:afonso}
Let $a,b,\delta>0$ ($a\ne b$) and let
$$
	B = \begin{pmatrix}
	\delta & 1 \\ 1 & 0
	\end{pmatrix},\qquad
	v = \begin{pmatrix} a\\ b\end{pmatrix},\qquad
	w = \begin{pmatrix} b\\ a\end{pmatrix}.
$$
We readily compute
$$
	d(v,w)^2 = \delta(a^2-b^2)^2,
$$
while
$$
	\|x(v)-x(w)\|^2 =
	(a\sqrt{\delta a^2+b^2}-b\sqrt{\delta b^2+a^2})^2
	\sim
	\delta^2(a^4-b^4)^2/4a^2b^2
	\mbox{ as }\delta\downarrow 0.
$$
Thus the ratio
$$
	\frac{d(v,w)}{\|x(v)-x(w)\|} \sim
	\frac{1}{\sqrt{\delta}}
	\frac{2ab}{a^2+b^2}\quad\mbox{as }\delta\downarrow 0
$$
can be arbitrarily large.  This example is essentially the worst possible,
as optimizing $\gamma$ in Lemma \ref{lem:basic} shows that
$d(v,w)^2\lesssim \max_{ij}b_{ij}\,\|x(v)-x(w)\|$ for $v,w\in B_2$.
\end{example}

\begin{rem}
\label{rem:singular}
While the above example illustrates conclusively that the desired 
inequality cannot hold in general when $B$ is not positive semidefinite, 
we also note that the failure point in this example appears to be very 
special.  The vectors $v$ and $w$, while far apart in the Euclidean 
distance, satisfy both $d(v,w)=0$ and $\|x(v)-x(w)\|=0$ when $\delta=0$.  
These points are therefore in some sense ``singular'' with respect to 
the geometry of $(B_2,d)$ and of $(B_*,\|\cdot\|)$ when $\delta=0$. 
Example~\ref{ex:afonso} shows that the comparison between the two 
geometries can fail near such singular points.  Numerical experiments 
suggest that such points are rather rare 
and that the inequality $d(v,w)\le 2\|x(v)-x(w)\|$ typically fails only 
in a very small subset of the unit ball.  We do not have a precise 
formulation of this idea, however.
\end{rem}

The phenomenon that is illustrated by Example~\ref{ex:afonso} is 
controlled in Lemma~\ref{lem:basic} by the addition of a second term 
that dominates the bound at the singular points of the geometry of $(B_2,d)$. The 
remarkable aspect of this second term is that it has a very suggestive 
interpretation: if the matrix $-B$ were positive semidefinite (which of 
course cannot be the case as $B$ has nonnegative entries), this would be 
the natural distance corresponding to Gaussian process defined by the 
convex hull of random variables $U_1,\ldots,U_d$ with $U\sim N(0,-B)$.  
By Lemma~\ref{lem:supg1}, the supremum of this Gaussian process would be 
of the same order as the second term of the last expression in Theorem 
\ref{thm:main1}, which would suffice to establish Conjecture~\ref{conj:1}.

While this intuition clearly cannot be implemented in this manner, it is 
nonetheless highly suggestive that the validity of 
Conjecture~\ref{conj:1} can ``almost'' be read off from the geometric 
structure described by Lemma~\ref{lem:basic}.  Unfortunately, we do not 
know how to optimally exploit this geometric structure.  In 
Theorem~\ref{thm:main3ext}, we have crudely forced $-B=B^--B^+$ to be 
positive definite by estimating it from above by $B^-$.  The problem 
with this approach is that the entries of $B^-$ can be much larger than 
the entries of $B$, which is the origin of the suboptimal second term in 
Corollary~\ref{cor:main3sharp}: there can in general be significant 
cancellation between $B^-$ and $B^+$ that our approach fails to exploit. 
The elimination of this inefficiency in the proof of 
Theorem~\ref{thm:main3ext} would be a significant step towards 
the resolution of Conjecture \ref{conj:1}.

We conclude by noting that there is no reason, in principle, to expect 
that a sharp bound on the expected supremum of a Gaussian process can 
always be obtained using the Slepian-Fernique inequality, as we have 
done in the proof of Theorem~\ref{thm:main1ext}.  In general, the 
connection between the supremum of a Gaussian process and the underlying 
geometry is described by the generic chaining method \cite{Tal14}.  
Unfortunately, even a geometric description along these lines
 of the trivial behavior of the supremum of a Gaussian process over a 
convex hull remains a long-standing open problem \cite[pp.\ 
50--51]{Tal14}, so that a direct application of generic chaining methods 
in the present setting appears to present formidable difficulties.

\subsection*{Acknowledgments}

Many of the results presented here were obtained while the author was a 
visitor at IMA in the spring semester of 2015 in the context of the 
annual program ``Discrete Structures: Analysis and Applications,'' and 
while the author attended the Oberwolfach workshop ``Probabilistic 
Techniques in Modern Statistics'' in May 2015. It is a pleasure to thank 
IMA and Oberwolfach, and in particular the organizers of these excellent 
programs, for their hospitality.  The author is grateful to Rafa\l\ 
Lata{\l}a, Afonso Bandeira, and Amirali Ahmadi for several interesting 
discussions, and to Afonso Bandeira for suggesting Example 
\ref{ex:afonso}.

An early draft of this paper was dedicated with great admiration to 
Evarist Gin\'e on the occasion of his 70th birthday.  I was immensely 
saddened to learn that Evarist unexpectedly passed away shortly after 
the completion of this draft, as is begrudgingly reflected in the 
dedication of the present version of this paper.


\end{document}